\documentclass[12pt,twoside,leqno,openany]{amsart}
\usepackage{amssymb,amsbsy,amsmath,amsfonts,amscd,times}
\usepackage{graphics,color,xy,xypic,footmisc,fancyhdr,multicol}
\usepackage{fancybox,graphicx,mathrsfs}
\usepackage[T1]{fontenc}
\usepackage{breqn}
\let\mathcal\mathscr
\newtheorem{theorem}{Theorem}
\newtheorem{lemma}{Lemma}

\newcommand{\C}{\mathbb{C}}

\newcommand{\R}{\mathbb{R}}

\newcommand{\ov}{\overline}
\newcommand{\LL}{\mathcal{L}}
\newcommand{\Lb}{\overline{\mathcal{L}}}

\newcommand{\RR}{\mathcal{R}}
\newcommand{\T}{\mathcal{T}}
\newcommand{\s}{\mathcal{S}}

\newcommand{\Sb}{\ov{\mathcal{S}}}
\newcommand{\zb}{\ov{z}}
\newcommand{\A}{{\sf a}}
\newcommand{\ab}{\ov{{\sf a}}}
\newcommand{\bb}{{\sf b}}
\newcommand{\bbb}{\ov{{\sf b}}}
\newcommand{\cc}{{\sf c}}

\newcommand{\dd}{{\sf d}}

\newcommand{\ee}{{\sf e}}

\newcommand{\ff}{{\sf f}}

\newcommand{\G}{{\sf g}}
\newcommand{\hh}{{\sf h}}

\newcommand{\kk}{{\sf k}}

\begin{document}

\setcounter{page}{97}

\title[]{
Canonical Cartan connection for $5$-dimensional CR-manifolds belonging to general class ${\sf III_2}$}

\author{Samuel Pocchiola}
\address{Samuel Pocchiola ---  D\'epartement de math\'ematiques, b\^atiment 425, Facult\'e des sciences d'Orsay,
Universit\'e Paris-Sud, F-91405 Orsay Cedex, france}
\email{samuel.pocchiola@math.u-psud.fr}

\maketitle

\bigskip
\section*{abstract}
We study the equivalence problem for CR-manifolds belonging to general class ${\sf III_2}$, 
i.e. the $5$-dimensional
CR-manifolds of CR-dimension $1$ and codimension $3$ whose CR-bundle satisfies a degeneracy condition which has been introduced in \cite{MPS}.
For such a CR-manifold $M$, we construct a canonical Cartan connection on a $6$-dimensional principal bundle $P$ on $M$.
This provides a complete set of biholomorphic invariants for $M$.

\section{Introduction}
As highlighted by Henri Poincar\'e \cite{Poincare} in 1907, the (local) biholomorphic equivalence problem between two submanifolds
$M$ and $M^{\prime}$ of $\C^N$
is to determine whether or not there exists a (local) biholomorphism $\phi$ of $\C^N$ such that $\phi(M) = M^{\prime}$.
Elie Cartan \cite{Cartan-1932, Cartan-1933} solved this problem for 
hypersurfaces $M^3 \subset \C^2$ in 1932, as he constructed a ``hyperspherical connection'' on such hypersurfaces
by using the powerful technique which is now referred to as Cartan's equivalence method.

Given a manifold $M$ and some geometric data specified on $M$, which usually appears as a $G$-structure on $M$ (i.e. a reduction of the bundle of coframes of $M$),
Cartan's equivalence method seeks to provide a principal bundle $P$ on $M$ together with a coframe $\omega$ of $1$-forms on $P$ which
is adapted to the geometric structure of $M$ in the following sense:
an isomorphism between two such
geometric structures $M$ and $M^{\prime}$ lifts to a unique isomorphism between $P$ and $P^{\prime}$ which sends $\omega$ on $\omega^{\prime}$.
The equivalence problem between $M$ and $M^{\prime}$ 
is thus reduced to an equivalence problem between $\{e\}$-structures, which is well understood \cite{Olver-1995, Sternberg}.

We recall that a  CR-manifold $M$ is a real manifold endowed with a 
subbundle $L$ of $\C \otimes TM$ of even rank $2n$ such that
\begin{enumerate}
\item{$L \cap \ov{L}$ = \{0\}}
\item{$L$ is formally integrable, i.e. $\big[ L, \, L \big] \subset L$}.
\end{enumerate}
The integer $n$ is the CR-dimension of $M$ and $k = \dim M - 2n$ is the codimension of $M$. 
In a recent attempt \cite{MPS} to solve the equivalence problem for CR-manifolds up to dimension $5$, it has been shown  
that one can restrict the study to six different general classes of CR-manifolds of dimension $\leq 5$, 
which have been referred to as general classes 
${\sf I}$, ${\sf II}$, ${\sf III}_1$, ${\sf III}_2$, ${\sf IV}_1$ and ${\sf IV}_2$. 
The aim of this paper is to provide a solution to the equivalence problem for CR-manifolds which belong to general class ${\sf III_2}$,
that is the CR-manifolds of dimension $5$ and of CR-dimension $1$ such that $\C \otimes TM$ is spanned by $L$, $\ov{L}$ and their Lie brackets up to order no 
less than $3$. 
More precisely, the following rank conditions hold:
\begin{equation*}
\begin{aligned}
3 &= \text{rank}_{\C} \left( L + \ov{L} + \big[ L, \, \ov{L} \big] \right), \\
4 & = \text{rank}_{\C} \left( L + \ov{L} + \big[ L, \, \ov{L} \big] + \big[ L, \big[ L, \, \ov{L} \big] \big] \right), \\
4 & = \text{rank}_{\C} \left( L + \ov{L} + \big[ L, \, \ov{L} \big] + \big[ L, \big[ L, \, \ov{L} \big] \big]+ 
\big [ \ov{L}, \big[ L, \, \ov{L} \big] \big] \right), \\ 
5 & =  \text{rank}_{\C} \left( L + \ov{L} + \big[ L, \, \ov{L} \big] + \big[ L, \big[ L, \, \ov{L} \big] \big] + \big[ \ov{L}, \big[ L, \, \ov{L} \big] \big]
+  \big[ L, \big[ L, \big[ L, \, \ov{L} \big] \big] \big] \right),
\end{aligned}
\end{equation*}
the third one beeing an exceptional degeneracy assumption. 

The main result of the present paper is the following:
\begin{theorem} \label{thm:intro}
Let $M$ be a CR-manifold belonging to general class ${\sf III_2}$.
There exists a 6-dimensional subbundle $P$ of the bundle of coframes $\C \otimes F(M)$ of $M$ and a coframe 
$\omega:=(\Lambda, \tau, \sigma, \rho, \zeta, \ov{\zeta})$ on $P$ such that any CR-diffeomorphism $h$ of $M$ lifts to a bundle isomorphism
$h^*$ of $P$ which satisfies $h^*(\omega) = \omega$. Moreover the structure equations of $\omega$ on $P$ are of the form:
\begin{equation*}
d \tau 
=
4 \, \Lambda \wedge \tau 
+
\mathfrak{J}_1\left. \tau \wedge \zeta \right. 
-
\mathfrak{J}_1\left. \tau \wedge \ov{\zeta} \right.
+
3  \, \mathfrak{J}_1 \left. \sigma \wedge \rho \right.
+
\left.\sigma \wedge \zeta \right.
+
\left. \sigma \wedge \ov{\zeta} \right.,
\end{equation*}
\begin{multline*}
d \sigma = 3 \, \Lambda \wedge \sigma  \\ 
+ \mathfrak{J}_2 \, \tau \wedge \rho 
+ \mathfrak{J}_3 \, \tau \wedge \zeta
+ \ov{\mathfrak{J}_3} \, \tau \wedge \ov{\zeta}
+  
\mathfrak{J}_4 \, \sigma \wedge \rho \\
- \frac{\mathfrak{J}_1}{2} \left. \sigma \wedge \zeta \right.
+
\frac{\mathfrak{J}_1}{2} \left. \sigma \wedge \ov{\zeta} \right. 
 +
\rho \wedge \zeta
+
\rho \wedge \ov{\zeta}
,
\end{multline*}
\begin{multline*}
d \rho
=
2 \Lambda \wedge \rho \\
+
\mathfrak{J}_5
\left. \tau \wedge \sigma \right.  
+
\mathfrak{J}_6
\left. \tau \wedge \rho \right.
+
\mathfrak{J}_7
\left. \tau \wedge \zeta \right.
+
\ov{\mathfrak{J}}_7
\left. \tau \wedge \ov{\zeta} \right. 
+
\mathfrak{J}_8 \left. \sigma \wedge \rho \right.
+
\mathfrak{J}_9
\left. \sigma \wedge \zeta \right. \\
+
\ov{\mathfrak{J}}_9
\left. \sigma \wedge \ov{\zeta} \right.
-
\frac{\mathfrak{J}_1}{2}
\left. \rho \wedge \zeta \right.  
+
 \frac{\mathfrak{J}_1}{2} \left. \rho \wedge \ov{\zeta} \right.
+
i \, \left. \zeta \wedge \ov{\zeta} \right.
,
\end{multline*}
\begin{multline*}
d \zeta
=
\Lambda \wedge \zeta \\
+
\mathfrak{J}_{10}
\left. \tau \wedge \sigma \right.  
+
\mathfrak{J}_{11}
\left. \tau \wedge \rho \right.
+
\mathfrak{J}_{12}
\left. \tau \wedge \zeta \right.
+
\mathfrak{J}_{13}
\left. \tau \wedge \ov{\zeta} \right. 
\\
+
\mathfrak{J}_{14}
\left. \sigma \wedge \rho \right.
+
\mathfrak{J}_{15}
\left. \sigma \wedge \zeta \right.
,\end{multline*}
\begin{equation*}
d \Lambda= \sum_{\nu \mu} 
X_{\nu \mu} \left. \nu \wedge \mu \right., \qquad \nu, \, \mu = \tau, \, \sigma, \, \rho,  \,  \zeta, \ov{\zeta},
\end{equation*}
where $\mathfrak{J}_{i}$, $X_{\nu \mu}$,
are functions on $P$.
\end{theorem}

The model manifold for this class is provided by the CR-manifold ${\sf N} \subset \C^3$  given by the equations:
\begin{equation*}
{\sf N}: \qquad \qquad
\begin{aligned}
w_1 & = \ov{w_1} + 2 \, i \, z \ov{z}, \\
w_2 & = \ov{w_2} + 2 \, i \, z \ov{z} \left( z + \ov{z} \right), \\
w_3 & = \ov{w_3} + 2 i \, z \zb \left( z^2 + \frac{3}{2} \, z \zb + \zb^2 \right),
\end{aligned}
\end{equation*}
Cartan's equivalence method has been applied to this model in \cite{pocchiola2}, where it has been shown that the coframe
$(\Lambda, \tau, \sigma, \rho, \zeta, \ov{\zeta})$ of theorem \ref{thm:intro} satisfy the simplified structure equations:
\begin{equation*}
\begin{aligned}
d \tau &= 4 \left. \Lambda \wedge \tau \right.
+ \left.\sigma \wedge \zeta \right.
+ \left. \sigma \wedge \ov{\zeta} \right., \\
d \sigma &= 3 \left. \Lambda \wedge \sigma \right. + \left. \rho \wedge \zeta \right. 
+ \left.  \rho \wedge \ov{\zeta} \right., \\
d \rho &  =  2 \left. \Lambda \wedge \rho \right. + i \, \left. \zeta \wedge \ov{\zeta} \right., \\
d \zeta &= \left. \Lambda \wedge \zeta \right., \\
d \ov{\zeta} &= \left. \Lambda \wedge \ov{\zeta} \right.,\\
d \Lambda &= 0,
\end{aligned}
\end{equation*}
corresponding to the case where the biholomorphic invariants $\mathfrak{J}_i$ vanish identically.
This result, together with the Lie algebra structure of the inifinitesimal CR-automorphisms of the model,
implies the existence of a Cartan connection on $M$, which we construct in section \ref{connection}.

We start in section \ref{G-structure} with the construction of a canonical $G$-structure $P^1$ on $M$, 
(e.g. a subbundle of the bundle
of coframes of $M$), 
which encodes the equivalence problem for $M$ under CR-automorphisms in the following sense:
a diffeomorphism 
\begin{equation*}
h: M \longrightarrow M
\end{equation*}
is a CR-automorphism of $M$ 
if and only if  
\begin{equation*}
h^* :P^1 \longrightarrow P^1
\end{equation*}
is a $G$-structure isomorphism of $P^1$.
We refer to \cite{MPS, Merker-2013-III, Merker-2013-IV} for details on the results summarized in this section
and to \cite{Sternberg} for an introduction to $G$-structures.
Section \ref{P1} is devoted to reduce successively $P^1$ to four subbundles:
\begin{equation*}
P^5 \subset P^4 \subset P^3 \subset P^2 \subset P^1,
\end{equation*}
which are still adapted to the biholomorphic equivalence problem for $M$. We use Cartan equivalence method, for which we refer to \cite{Olver-1995}.
Eventually a Cartan connection is constructed on $P^5$ in section \ref{connection}.

\section{Initial G-structure}\label{G-structure}

Let $M$ be a CR-manifold belonging to general class ${\sf III_2}$ and $\LL$ be a local generator of the CR-bundle $L$ of $M$.
As $M$ belongs to general class ${\sf III_2}$, the three vector fields $\T$, $\s$, $\RR$, defined by:
\begin{equation*}
\begin{aligned}
\T & := i \, \big[ \LL, \Lb \big], \\
\s & := \big[ \LL, \T \big],\\
\RR & := \big[ \LL, \T \big],
\end{aligned}
\end{equation*}
are such that the following biholomorphic invariant conditions hold:
\begin{alignat*}{2}
3 &= \text{rank}_{\C}\left(\LL, \Lb, \T \right), &\qquad \qquad
4 &=  \text{rank}_{\C}\left( \LL, \Lb, \T , \s\right),\\
4 &=  \text{rank}_{\C}
\left(\LL, \Lb, \T , \s, \Sb \right)
,  &\qquad \qquad
5 &= \text{rank}_{\C}
\left(\LL, \Lb, \T , \s, \RR \right)
.
\end{alignat*}
As a result there exist two functions $A$ and $B$ such that:
\begin{equation*}
\Sb = A \cdot \T + B \cdot \s.
\end{equation*}
From the fact that $ \ov{\Sb} = \s,$ the functions $A$ and $B$ satisfy the relations:
\begin{equation*}
\begin{aligned}
B \ov{B} & = 1, \\
\ov{A} +  \ov{ B} A &= 0.
\end{aligned}
\end{equation*}
There also exist three functions $E$, $F$, $G$,
such that:
\begin{equation*}
\big[\LL,\RR \big]
 = E \cdot \T + F \cdot \s + G \cdot \RR.
\end{equation*}
The five functions $A$, $B$, $E$, $F$, $G$ appear to be fundamental as all other Lie brackets between the 
vector fields $\LL$, $\Lb$, $\T$, $\s$ and $\RR$ can be expressed in terms of these five functions and their
$\{ \LL, \Lb\}$-derivatives.

In the case of an embedded CR-manifold $M \subset \C^4$, we can give an explicit formula for the fundamental vector field $\LL$, and hence
for the functions  $A$, $B$, $P$, $Q$, in terms of 
a graphing function of $M$. We refer to \cite{Merker-2013-V} for details on this question.
Let us just mention that the submanifold $M \subset \C^4$ 
is represented in local coordinates:
\begin{equation*}
(z,w_1,w_2,w_3) = (x + i \, y, \, u_1 + i\, v_1, \,u_2 + i\, v_2, \, u_3 + i\, v_3),
\end{equation*}
as a graph:
\begin{equation*}
\begin{aligned}
v_1 &= \phi_1(x, y, u_1, u_2, u_3), \\
v_2 &= \phi_2(x, y, u_1, u_2, u_3), \\
v_3 &= \phi_3(x, y, u_1, u_2, u_3).
\end{aligned}
\end{equation*}
There exists a unique local generator $\LL$ of $T^{1,0}M$ of the form:
\begin{equation*}
\LL= \frac{\partial}{\partial z} 
+ A^1 \, \frac{\partial}{\partial u_1} 
+  A^2 \, \frac{\partial}{\partial u_2} 
+ A^3 \, \frac{\partial}{\partial u_3},  
\end{equation*}
having conjugate:
\begin{equation*}
\Lb= \frac{\partial}{\partial \zb} 
+ \ov{A^1} \, \frac{\partial}{\partial u_1} 
+  \ov{A^2} \, \frac{\partial}{\partial u_2} 
+ \ov{A^3} \, \frac{\partial}{\partial u_3},  
\end{equation*}
which is a generator of $T^{0,1}M$. The explicit expressions of the functions $A^1$, $A^2$ and $A^3$ in terms of $\phi$
can be found in  \cite{Merker-2013-V}.

Returning to the general case of abstract CR-manifolds, let 
\begin{equation*}
\omega_0 :=\left(\tau_0, \sigma_0, \rho_0, \zeta_0, \ov{\zeta}_0 \right)
\end{equation*}
be the dual coframe of 
$\left(\RR, \s, \T, \LL, \Lb \right)$.  We have:
\begin{lemma}{\cite{Merker-2013-IV}.}
The structure equations enjoyed by $\omega_0$ are of the form: 
\begin{dgroup*}
\begin{dmath*}
d \tau_0 = T  \left. \tau_0 \wedge \sigma_0 \right. + Q  \left. \tau_0 \wedge \rho_0 \right.
+ K  \left. \tau_0 \wedge \zeta_0 \right. + G \left. 
\tau_0 \wedge \zeta_0 \right.
+ N  \left. \sigma_0 \wedge \rho_0 \right. + \left. \sigma_0 \wedge \zeta_0 \right. 
+ B  \left. \sigma_0 \wedge \ov{\zeta}_0 \right., 
\end{dmath*}
\begin{dmath*}
d\sigma_0 = S \left. \tau_0 \wedge \sigma_0  \right.+ P \left. \tau_0 \wedge 
\rho_0 \right. + F \left.  \tau_0 \wedge \zeta_0 \right.+ J \left. \tau_0 \wedge \ov{\zeta}_0 \right.
+ M \left. \sigma_0 \wedge \rho_0 \right.
+
\left( \LL(B) + A \right) \left. \sigma_0 \wedge \ov{\zeta}_0 \right.+ B \left. \rho_0 \wedge \ov{\zeta}_0 
\right.
+ \left. \rho_0 \wedge \zeta_0 \right.,
\end{dmath*}
\begin{dmath*}
d\rho_0  = R \left. \tau_0 \wedge \sigma_0 \right. 
+ 
O \left.  \tau_0 \wedge \rho_0  \right. + H \left. \tau_0 \wedge \zeta_0 \right.
+ E \left. \tau_0 \wedge \zeta_0 \right. + L \left. \sigma_0 \wedge \rho_0 \right.
+ \LL(A) \left.  \sigma_0 \wedge \ov{\zeta}_0 \right. + A \left. \rho_0 \wedge \ov{\zeta}_0 \right.
+ i \left. \zeta_0 \wedge \ov{\zeta}_0 \right., 
\end{dmath*}
\begin{dmath*}
d \zeta_0 = 0,
\end{dmath*}
\begin{dmath*}
d \ov{\zeta}_0 = 0,
\end{dmath*} 
\end{dgroup*}
where the twelve functions:
\begin{equation*}
H,\, J,\, K,\, L,\, M,\, N,\, O,\, P,\, Q,\, R,\, S,\, T,  
\end{equation*}
can be expressed 
in terms of the five fundamental functions:
\begin{equation*}
A,\, B,\, E,\, F,\, G,
\end{equation*}
and their $\{ \LL, \Lb\}$-derivatives.
\end{lemma}

Let $h: \, M \longrightarrow M$ be a CR-automorphism of $M$. As we have 
\begin{equation*}
h_* \left(L \right) = L,  
\end{equation*}
there exists a non-vanishing complex-valued function $\A$ on $M$ such that:
\begin{equation*}
h_* \left( \LL \right) = \A \, \LL.
\end{equation*}
From the definition of $\T$, $\s$, $\RR$ and the invariance 
\begin{equation*}
h_* \left( \big[ X, Y \big] \right) = \big[ h_*(X), h_*(Y) \big]
\end{equation*}
for any vector fields $X$, $Y$ on $M$, we easily get the existence of eight functions
\begin{equation*}
\bb, \cc , \dd , \ee , \ff, \G , \hh, \kk: M \longrightarrow \C,
\end{equation*}
such that 
\begin{equation*}
h_*
\begin{pmatrix}
\LL \\
\Lb \\
\T \\
\s \\
\RR 
\end{pmatrix}
=
\begin{pmatrix}
\A & 0 & 0 & 0 & 0 \\
 0 & \ab & 0 & 0 & 0 \\
\bb & \bbb & \A \ab & 0 &0 \\
\ee & \dd & \cc & \A^2 \ab & 0 \\
\kk & \hh & \G & \ff & \A^3 \ab
\end{pmatrix}
\cdot
\begin{pmatrix}
\LL \\
\Lb \\
\T \\
 \s \\
 \RR 
\end{pmatrix}.
\end{equation*}

This is summarized in the following lemma: 
\begin{lemma}{\cite{Merker-2013-III}.}
\label{lemma:matrix}
Let $h: \, M \longrightarrow M$ a CR-automorphism of $M$ and let $G_1$ 
be the subgroup of ${\sf GL}_5(\C)$:
\begin{equation*}
G_1:= \left\{
\begin{pmatrix}
{\A^3} \ab & 0 & 0 & 0 & 0 \\
\ff  & \A^2 \ab & 0 & 0 & 0 \\
\G & \cc & \A \ab & 0 & 0 \\
\hh & \dd & \bb & \A & 0 \\
\kk & \ee & \bbb & 0 & \ab
\end{pmatrix}
, \, 
\A \in \C \setminus{\{0\}}, \,
 \bb, \cc, \dd, \ee, \ff, \G, \hh, \kk \in \C \right \}.
\end{equation*}
Then the pullback $\omega$ of $\omega_0$ by $h$,
$\omega:= h^* \omega_0,$
satisfies:
\begin{equation*}
\omega = g \cdot \omega_0,
\end{equation*}
where $g$ is smooth (locally defined) function 
$M \stackrel{g}{\longrightarrow} G_1$.
\end{lemma}

Let $P^1$ be the $G_1$-structure on $M$ defined by the coframes $\omega$ of the form
\begin{equation*}
\omega:= g \cdot \omega_0, \qquad g \in G_1
\end{equation*}
The next section is devoted  to construct four
subgroups of $G_1$:
\begin{equation*}
G_5 \subset G_4 \subset G_3 \subset G_2 \subset G_1,
\end{equation*}
and four $G_i$-structures on $M$: 
\begin{equation*}
P^5 \subset P^4 \subset P^3 \subset P^2 \subset P^1,
\end{equation*}
which are adapted to the biholomorphic equivalence problem for $M$
in the sense that a diffeomorphism $h$ of $M$ is a CR-automorphism if and only if  $h^*$ is a
$G_i$-structure isomorphism of $P^i$.

\section{Reductions of $P^1$}\label{P1}

The coframe $\omega_0$ gives a natural (local) trivialisation $P^1 \stackrel{tr} \longrightarrow M \times G_{1}$ from which
we may consider any differential form on $M$  
(resp. $G_1$) as a differential form on $P^1$ through the pullback by the first
(resp. the second) component of $tr$.
With this identification, 
the structure equations of $P^1$ are naturally obtained by the formula:
\begin{equation} 
\label{eq:lifted}
d \omega =  dg \cdot g^{-1} \wedge \omega + g \cdot d \omega_{0}.
\end{equation}
The term $ g \cdot d \omega_{0}$  contains the so-called torsion coefficients of $P^1$.
A $1$-form $\widetilde{\alpha}$ on $P^1$ is called a modified Maurer-Cartan form if its restriction to any fiber of $P^1$ 
is a Maurer-Cartan form of $G_1$, or equivalently, if it is of the form:
\begin{equation*}
\widetilde{\alpha} := \alpha - x_{\tau}\, \tau - x_{\sigma} \, \sigma  -x_{\rho} \, \rho 
- x_{\zeta} \, \zeta \, -
x_{\overline{\zeta}} \, \overline{\zeta},
\end{equation*}
where $x_{\sigma}$, $x_{\rho}$, $x_{\zeta}$, $x_{\ov{\zeta}}$, are arbitrary complex-valued functions on $M$ and
where $\alpha$
is a Maurer-Cartan form of $G_1$.

A basis for the Maurer-Cartan forms of $G_1$ is given by the following $1$-forms:
\begin{equation*}
\begin{aligned}
\alpha^1  &: = {\frac {{ d\A}}{\A}}, \\
\alpha^2 &:= -{\frac {\bb{ d\A}}{{\A}^{2}{ \ab}}}+{\frac {{ d\bb}}{\A{ \ab}}}, \\
\alpha^3  &:=-{\frac {\cc{ d\A}}{{ \ab}\,{\A}^{3}}}-{\frac {\cc{ d\ab
}}{{{ \ab}}^{2}{\A}^{2}}} + { \frac {{ d\cc}}{{\A}^{2}{ \ab}}}, \\
\alpha^4 & :=-{\frac { \left( \dd\A{ \ab}-\bb\cc \right) { d\A}}{{\A}^{4}{
{ \ab}}^{2}}}-{\frac {\cc{ d\bb}}{{\A}^{3}{{ \ab}}^{2}}}+{\frac {{
 d\dd}}{{\A}^{2}{ \ab}}}, \\
\alpha^5 & :=-{\frac { \left( \ee\A{ \ab}-{ \bbb}\,\cc \right) { d\ab}
}{{\A}^{3}{{ \ab}}^{3}}}-{\frac {\cc{ d\bbb}}{{\A}^{3}{{ \ab}}^{2}}}+
{\frac {{ d\ee}}{{\A}^{2}{ \ab}}}, \\
\alpha^6 & :=-2\,{\frac {\ff{ d\A}}{{ \ab}\,{\A}^{4}}}-{\frac {\ff{ 
d\ab}}{{\A}^{3}{{ \ab}}^{2}}}+{\frac {{ d\ff}}{{ \ab}\,{\A}^{3}}}, \\
\alpha^7 & :=-{\frac { \left( \G{\A}^{2}{ \ab}-\cc\ff \right) { d\A}}{{{
 \ab}}^{2}{\A}^{6}}}-{\frac { \left( \G{\A}^{2}{ \ab}-\cc\ff \right) {
 d\ab}}{{{ \ab}}^{3}{\A}^{5}}}-{\frac {\ff{ d\cc}}{{\A}^{5}{{ \ab}}
^{2}}}+{\frac {{ d\G}}{{ \ab}\,{\A}^{3}}}, \\
\alpha^8 &: =-{\frac { \left( \hh{\A}^{3}{{ \ab}}^{2}- \dd \ff\A{ \ab}-\bb \G{\A}
^{2}{ \ab}+\bb\cc\ff \right) { d\A}}{{\A}^{7}{{ \ab}}^{3}}}-{\frac {
 \left( \G{\A}^{2}{ \ab}-\cc\ff \right) { d\bb}}{{\A}^{6}{{ \ab}}^{3}}}-
{\frac {\ff{ d\dd}}{{\A}^{5}{{ \ab}}^{2}}}+{\frac {{ d\hh}}{{ \ab}
\,{\A}^{3}}},\\
\alpha^9 & :=-{\frac { \left( \kk{\A}^{3}{{ \ab}}^{2}-\ee\ff\A{ \ab}-{ 
\bbb}\,\G{\A}^{2}{ \ab}+{ \bbb}\,\cc\ff \right) { d\ab}}{{\A}^{6}{{ \ab}
}^{4}}}-{\frac { \left( \G{\A}^{2}{ \ab}-\cc\ff \right) { d\bbb}}{{\A}^{6}
{{ \ab}}^{3}}}-{\frac {\ff{ d\ee}}{{\A}^{5}{{ \ab}}^{2}}}+{\frac {{
 d\kk}}{{ \ab}\,{\A}^{3}}},
\end{aligned}
\end{equation*}
together with their conjugates.

We derive the structure equations of $P^1$ from the relations (\ref{eq:lifted}). The expression 
of $d \tau$ is:
\begin{multline*}
d \tau = 3 \left.\alpha^1 \wedge \tau \right. + \left.\ov{\alpha^1} \wedge \tau \right. \\
+ T^{\tau}_{\tau \sigma} \left. \tau \wedge \sigma \right.
+ T^{\tau}_{\tau \rho} \left. \tau \wedge \rho \right.
+ T^{\tau}_{\tau \zeta} \left. \tau \wedge \zeta \right. \\
+ T^{\tau}_{\tau \ov{\zeta}} \left. \tau \wedge \ov{\zeta} \right.
+ T^{\tau}_{\sigma \rho} \left.\sigma \wedge \rho \right.
+ \left.\sigma \wedge \zeta \right.
-  \frac{\A}{\ab} \, B \, \left. \sigma \wedge \ov{\zeta} \right. 
\end{multline*}

The coefficient 
\begin{equation*}
\frac{\A }{\ab} \, B, 
\end{equation*}
which can not be absorbed for any choice of the modified Maurer-Cartan form $\widetilde{\alpha}^1$,
is referred to as an essential torsion coefficient.
From standard results on Cartan theory (see \cite{Olver-1995, Sternberg}), a diffeomorphism of $M$
is an isomorphism of the $G_1$-structure $P^1$ if and only if it is an isomorphism of the reduced 
bundle $P^2 \subset P^1$ consisting
of those coframes $\omega$ on $M$ such that  
\begin{equation*}
\frac{\A }{\ab} \, B  = 1. 
\end{equation*}
This is equivalent to the normalization: 
\begin{equation*}
 \ab = \A B.
\end{equation*}

A coframe $\omega \in P^2$ is related to the coframe $\omega_0$ by the relations:
\begin{alignat*}{2}
\tau & = \A^4 \, B \, \tau_0,  & \qquad \qquad 
\sigma &= \ff \, \tau_0 + \A^3 \, B \, \sigma_0, \\
\rho &= \G \, \tau_0 + \cc \, \sigma_0 + \A^2 \, B \, \rho_0, & \qquad \qquad
\zeta & = \hh \, \tau_0 + \dd \,  \sigma_0 + \bb \, \rho_0 + \A \, \zeta_0,\\
\ov{\zeta} & = \kk \, \tau_0 + \ee \, \sigma_0 + \bbb \, \rho_0 + \A \, B \, \ov{\zeta}_0,
\end{alignat*}
which are equivalent to:
\begin{alignat*}{2}
\tau & = \left. \A^\prime \right.^4 \, \tau_1, & \qquad \qquad
\sigma &= \ff^\prime \, \tau_1 + \left. \A^\prime \right.^3 \, \sigma_1, \\
\rho &= \G^\prime \, \tau_1 + \cc^\prime \, \sigma_1 + \left. \A^\prime \right.^2 \, \rho_1, & \qquad \qquad
\zeta & = \hh^\prime \, \tau_1 + \dd^\prime \,  \sigma_1 + \bb \, \rho_1 +  \A^\prime \, \zeta_1,\\
\ov{\zeta} & = \kk^\prime \, \tau_1 + \ee^\prime \, \sigma_1 + \bbb \, \rho_1 +  \A^{\prime} \, \ov{\zeta}_1,
\end{alignat*} 
where:
\begin{equation*}
\tau_1:= \frac{\tau_0}{B}, \qquad
\sigma_1:  = \frac{\sigma_0}{B^{\frac{1}{2}}}, \qquad  \rho_1  = \rho_0, \qquad 
\zeta_1 : = \frac{\zeta_0}{B^{\frac{1}{2}}},
\end{equation*}
and
\begin{equation*}
x^{\prime} := \left\{ \begin{aligned}
&x \cdot B^{\frac{1}{2}}, \qquad  \qquad &\text{for}  \quad x = \A, \, \cc, \,  \dd, \, \ee, \\
&x \cdot B, \qquad \qquad  & \text{for}  \quad x = \ff, \, \G, \,  \hh, \, \kk.
\end{aligned}
\right.
\end{equation*}

We notice that $\A^{\prime}$ is a real parameter, and that $\tau_{1}$ is a real $1$-form.
Let $\omega_1$ be the coframe $\omega_1:=\left(\tau_1, \sigma_1, \rho_1, \zeta_1, \ov{\zeta}_1 \right)$, 
and $G_2$ be the subgroup of $G_1$:
\begin{equation*}
G_2:= \left\{
\begin{pmatrix} 
{\A^4}  & 0 & 0 & 0 & 0 \\
\ff  & \A^3 & 0 & 0 & 0 \\
\G & \cc & \A^2 & 0 & 0 \\
\hh & \dd & \bb & \A & 0 \\
\kk & \ee & \bbb & 0 & \A
\end{pmatrix}
, \, 
\A \in \R \setminus \{0\}, \,
\bb, \cc, \dd, \ee, \ff, \G, \hh, \kk \in \C \right \}.
\end{equation*}  
A coframe $\omega$ on $M$ belongs to $P^2$ if and only if there is a local function
$g: M \stackrel{g}{\longrightarrow}  G_2$ such that $\omega = g \cdot \omega_1$, namely $P^2$ is a $G_2$ structure on $M$. 

The Maurer-Cartan forms of $G_2$ are given by:
\begin{equation*}
\begin{aligned}
\beta^1  &: = {\frac {{ d\A}}{\A}}, \\
\beta^2 &:= -{\frac {\bb d\A }{{\A}^{3}}}+{\frac {{ d\bb}}{\A^2}}, \\
\beta^3  &:=- 2 \, {\frac {\cc  d\A }{{\A}^{4}}} + { \frac {{ d\cc}}{{\A}^{3}}}, \\
\beta^4 & =-{\frac { \left( \dd\A^2 -\bb \cc \right) { d\A}}{{\A}^{6}}}-{\frac {\cc{ d\bb}}{{\A}^{5}}}+{\frac {{
 d\dd}}{{\A}^{3}}}, \\
\beta^5 & =-\frac { \left( \ee\A^2-{ \bbb}\,\cc \right)  d\A
}{\A^6}-\frac {\cc{ d\bbb}}{\A^{5}}+
\frac { d\ee}{\A^{3}}, \\
\beta^6 & = -3 \,{\frac {\ff{ d\A}}{{\A}^{5}}} + {\frac {{ d\ff}}{{\A}^{4}}}, \\
\beta^7 & =- 2 \, {\frac { \left( \G{\A}^{3}-\cc \ff \right) 
{ d\A}}{\A^{8}}} -{\frac {\ff{ d\cc}}{{\A}^{7}}}+{\frac {{ d\G}}{{\A}^{4}}}, \\
\beta^8 & =-{\frac { \left( \hh{\A}^{5}- \dd \ff\A^2-\bb \G{\A}
^{3}+\bb\cc\ff \right) { d\A}}{{\A}^{10}}}-{\frac {
 \left( \G{\A}^{3}-\cc\ff \right) { d\bb}}{{\A}^{9}}}-
{\frac {\ff{ d\dd}}{\A^{7}}}+{\frac {{ d\hh}}{
{\A}^{4}}},\\
\beta^9 & =-{\frac { \left( \kk{\A}^{5}-\ee\ff\A^2-{ 
\bbb}\,\G{\A}^{3}+{ \bbb}\,\cc\ff \right) { d\A}}{{\A}^{10}}}-{\frac { \left( \G{\A}^{3}-\cc\ff \right) { d\bbb}}{{\A}^{9}}}
-
{\frac {\ff{ d\ee}}{{\A}^{5}{{ \ab}}^{2}}}+{\frac {{
 d\kk}}{{\A}^{4}}},
\end{aligned}
\end{equation*}
together with $\ov{\beta^i}, \,\,\, i=2 \dots 9$.

Using formula (\ref{eq:lifted}), we get the structure equations of $P^2$:
\begin{multline*}
d \tau 
=
4 \, \beta^1 \wedge \tau \\
+
U^{\tau}_{\tau \sigma} \, \tau \wedge \sigma
+
U^{\tau}_{\tau \rho} \, \tau \wedge \rho
+
U^{\tau}_{\tau \zeta} \, \tau \wedge \zeta 
+
U^{\tau}_{\tau \ov{\zeta}} \, \tau \wedge \ov{\zeta} \\
+
U^{\tau}_{\sigma \rho} \, \sigma \wedge \rho
+
\sigma \wedge \zeta
+
\sigma \wedge \ov{\zeta},
\end{multline*}

\begin{multline*}
d \sigma = 3 \, \beta^1 \wedge \sigma + \beta^6 \wedge \tau \\
+
U^{\sigma}_{\tau \sigma} \left. \tau \wedge \sigma \right.
+
U^{\sigma}_{\tau \rho} \left. \tau \wedge \rho \right.
+
U^{\sigma}_{\tau \zeta} \left. \tau \wedge \zeta \right. \\
+
U^{\sigma}_{\tau \ov{\zeta}} \left. \tau \wedge \ov{\zeta} \right.
+
U^{\sigma}_{\sigma \rho} \left. \sigma \wedge \rho \right.
+
U^{\sigma}_{\sigma \zeta} \left. \sigma \wedge \zeta \right. \\
+
U^{\sigma}_{\sigma \ov{\zeta}} \left. \sigma \wedge \ov{\zeta} \right.
+
\rho \wedge \zeta
+
\rho \wedge \ov{\zeta}
\end{multline*}

\begin{multline*}
d \rho
=
2 \beta^{1} \wedge \rho + \beta^3 \wedge \sigma + \beta^7 \wedge \tau \\
+
U^{\rho}_{\tau \sigma} \, \tau \wedge \sigma
+
U^{\rho}_{\tau \rho} \, \tau \wedge \rho
+
U^{\rho}_{\tau \zeta} \, \tau \wedge \zeta 
+
U^{\rho}_{\tau \overline{\zeta}} \, \rho \wedge \overline{\zeta} 
+
U^{\rho}_{\sigma \rho} \, \sigma \wedge \rho \\
+
U^{\rho}_{\sigma \zeta} \, \sigma \wedge \zeta
+
U^{\rho}_{\sigma \ov{\zeta}} \, \sigma \wedge \ov{\zeta} 
+
U^{\rho}_{\rho \zeta} \, \rho \wedge \zeta  
+
U^{\rho}_{\rho \ov{\zeta}} \, \rho \wedge \ov{\zeta} 
+
i \, \zeta \wedge \ov{\zeta}
,\end{multline*}

\begin{multline*}
d \zeta
=
{\beta}^{1} \wedge \zeta + {\beta}^{2} \wedge \rho + {\beta}^{4} \wedge \sigma + \beta^8 \wedge \tau \\
+
U^{\zeta}_{\tau \sigma} \, \tau \wedge \sigma
+
U^{\zeta}_{\tau \rho} \, \tau \wedge \rho
+
U^{\zeta}_{\tau \zeta} \, \tau \wedge \zeta 
+
U^{\zeta}_{\tau \ov{\zeta}} \, \tau \wedge \ov{\zeta} \\
+
U^{\zeta}_{\sigma \rho} \, \sigma \wedge \rho
+
U^{\zeta}_{\sigma \zeta} \, \sigma \wedge \zeta 
+
U^{\zeta}_{\sigma \ov{\zeta}} \, \sigma \wedge \ov{\zeta} 
+
U^{\zeta}_{\rho \zeta} \, \rho \wedge \zeta \\
+
U^{\zeta}_{\rho \overline{\zeta}} \, \rho \wedge \overline{\zeta}
+
U^{\zeta}_{\zeta \overline{\zeta}} \, \zeta \wedge \overline{\zeta}
.\end{multline*}
Introducing the modified Maurer-Cartan forms:
\begin{equation*}
\widetilde{\beta}^i= \beta^i  -y_{\tau}^i \, \tau - y_{\sigma} \, \sigma - y_{\rho}^i \, \rho  - y_{\zeta}^i \, \zeta \, 
- y_{\overline{\zeta}}^i \, \overline{\zeta},
\end{equation*}
the structure equations rewrite:
\begin{multline*}
d \tau 
=
4 \, \widetilde{\beta}^1 \wedge \tau \\
+
\left( U^{\tau}_{\tau \sigma} - 4 \, y^1_{\sigma} \right) \, \left. \tau \wedge \sigma \right.
+
\left( U^{\tau}_{\tau \rho} - 4 \, y^1_{\rho} \right) \, \left. \tau \wedge \rho \right. \\
+
\left(U^{\tau}_{\tau \zeta} - 4 \, y^1_{\zeta} \right) \, \left. \tau \wedge \zeta \right. 
+
\left(U^{\tau}_{\tau \ov{\zeta}} - 4 \,  y^1_{\ov{\zeta}} \right) \, \left. \tau \wedge \ov{\zeta} \right. \\
+
U^{\tau}_{\sigma \rho} \, \left. \sigma \wedge \rho \right.
+
\left.\sigma \wedge \zeta \right.
+
\left. \sigma \wedge \ov{\zeta} \right.,
\end{multline*}

\begin{multline*}
d \sigma = 3 \, \left. \widetilde{\beta}^1 \wedge \sigma \right. + \left. \widetilde{\beta}^6 \wedge \tau \right. \\
+
\left( U^{\sigma}_{\tau \sigma} + 3 \, y^1_{\tau} - y^6_{\sigma} \right) \, \left. \tau \wedge \sigma \right.
+
\left( U^{\sigma}_{\tau \rho} -  y^6_{\rho} \right) \, \left. \tau \wedge \rho \right. \\
+
\left( U^{\sigma}_{\tau \zeta} - y^6_{\zeta} \right) \, \left. \tau \wedge \zeta \right. 
+
\left( U^{\sigma}_{\tau \ov{\zeta}} - y^6_{\ov{\zeta}} \right) \, \left. \tau \wedge \ov{\zeta} \right. \\
+
\left( U^{\sigma}_{\sigma \rho} - 3 \, y^1_{\rho} \right) \, \left. \sigma \wedge \rho \right. 
+
\left( U^{\sigma}_{\sigma \zeta} -3 \,y^1_{\zeta} \right) \, \left. \sigma \wedge \zeta \right. \\
+
\left( U^{\sigma}_{\sigma \ov{\zeta}}  -3 \,y^1_{\ov{\zeta}} \right) \, \left. \sigma \wedge \ov{\zeta} \right.
+
\rho \wedge \zeta
+
\rho \wedge \ov{\zeta}
\end{multline*}

\begin{multline*}
d \rho
=
2 \widetilde{\beta}^{1} \wedge \rho + \widetilde{\beta}^3 \wedge \sigma + \widetilde{\beta}^7 \wedge \tau \\
+
\left( U^{\rho}_{\tau \sigma} + y^3_{\tau} - y^7_{\sigma} \right) \, \left. \tau \wedge \sigma \right.
+
\left( U^{\rho}_{\tau \rho} + 2 \, y^1_{\tau} - y^7_{\rho} \right) \, \left. \tau \wedge \rho \right. \\
+
\left( U^{\rho}_{\tau \zeta} - y^7_{\zeta} \right) \, \left. \tau \wedge \zeta \right. 
+
\left( U^{\rho}_{\tau \overline{\zeta}} - y^7_{\ov{\zeta}} \right) \, \left. \rho \wedge \overline{\zeta} \right. \\
+
\left( U^{\rho}_{\sigma \rho} + 2 \, y^1_{\sigma} - y^3_{\rho} \right) \, \left. \sigma \wedge \rho \right.
+
\left( U^{\rho}_{\sigma \zeta} - y^3_{\zeta} \right) \, \left. \sigma \wedge \zeta \right. \\
+
\left( U^{\rho}_{\sigma \ov{\zeta}} - y^3_{\ov{\zeta}} \right) \, \left. \sigma \wedge \ov{\zeta} \right.
+
\left( U^{\rho}_{\rho \zeta} - 2 \, y^1_{\zeta} \right) \, \left. \rho \wedge \zeta \right.  \\ 
+
\left( U^{\rho}_{\rho \ov{\zeta}} - 2 \, y^1_{\ov{\zeta}} \right) \, \left. \rho \wedge \ov{\zeta} \right.
+
i \, \left. \zeta \wedge \ov{\zeta} \right.
,\end{multline*}

\begin{multline*}
d \zeta
=
\widetilde{\beta^{1}} \wedge \zeta + \widetilde{\beta^{2}} \wedge \rho + \widetilde{\beta^{4}} \wedge \sigma + \widetilde{\beta^8} \wedge \tau \\
+
\left( U^{\zeta}_{\tau \sigma} + y^4_{\tau}- y^8_{\sigma} \right) \, \left. \tau \wedge \sigma \right.
+
\left( U^{\zeta}_{\tau \rho} + y^2_{\tau} - y^8_{\rho} \right) \, \left. \tau \wedge \rho \right. \\
+
\left( U^{\zeta}_{\tau \zeta} + y^1_{\tau} - y^8_{\zeta} \right) \, \left. \tau \wedge \zeta \right.
+
\left( U^{\zeta}_{\tau \ov{\zeta}} - y^8_{\ov{\zeta}} \right) \, \left. \tau \wedge \ov{\zeta} \right. \\
+
\left( U^{\zeta}_{\sigma \rho} + y^2_{\sigma}- y^4_{\rho} \right) \, \left. \sigma \wedge \rho \right.
+
\left( U^{\zeta}_{\sigma \zeta} + y^1_{\sigma} - y^4_{\zeta} \right) \, \left. \sigma \wedge \zeta \right. \\ 
+
\left( U^{\zeta}_{\sigma \ov{\zeta}} - y^4_{\ov{\zeta}} \right) \, \left. \sigma \wedge \ov{\zeta} \right. 
+
\left( U^{\zeta}_{\rho \zeta} + y^1_{\rho} - y^2_{\zeta} \right) \, \left. \rho \wedge \zeta \right. \\
+
\left( U^{\zeta}_{\rho \overline{\zeta}} - y^2_{\ov{\zeta}} \right) \, \left. \rho \wedge \overline{\zeta} \right.
+
\left( U^{\zeta}_{\zeta \overline{\zeta}} - y^1_{\ov{\zeta}} \right) \, \left. \zeta \wedge \overline{\zeta} \right.
.\end{multline*}

We get the following absorbtion equations:

\begin{alignat*}{3}
4 \, y^1_{\sigma} &=U^{\tau}_{\tau \sigma},
& \qquad \qquad  4 \, y^1_{\rho}  &=  U^{\tau}_{\tau \rho} ,
& \qquad \qquad  4 \, y^1_{\zeta} &=  U^{\tau}_{\tau \zeta} , \\
4 \,  y^1_{\ov{\zeta}} &= U^{\tau}_{\tau \ov{\zeta}}, 
& \qquad \qquad - 3 \, y^1_{\tau} + y^6_{\sigma}  & =U^{\sigma}_{\tau \sigma},  & 
\qquad \qquad y^6_{\rho}   & =  U^{\sigma}_{\tau \rho} , \\
y^6_{\zeta}  &= U^{\sigma}_{\tau \zeta},  & \qquad \qquad  
y^6_{\ov{\zeta}}  & =   U^{\sigma}_{\tau \ov{\zeta}}, & \qquad \qquad 
3 \, y^1_{\rho} & =  U^{\sigma}_{\sigma \rho}, \\  
3 \,y^1_{\zeta} & =  U^{\sigma}_{\sigma \zeta}, & \qquad \qquad
3 \,y^1_{\ov{\zeta}} & =  U^{\sigma}_{\sigma \ov{\zeta}}, & \qquad \qquad
-y^3_{\tau} + y^7_{\sigma}  & = U^{\rho}_{\tau \sigma},  \\ 
-2 \, y^1_{\tau} + y^7_{\rho} &=  U^{\rho}_{\tau \rho}, & \qquad \qquad 
y^7_{\zeta}  & = U^{\rho}_{\tau \zeta}, &\qquad \qquad   
y^7_{\ov{\zeta}}  &=  U^{\rho}_{\tau \overline{\zeta}}, \\
-2 \, y^1_{\sigma} + y^3_{\rho}  &= U^{\rho}_{\sigma \rho}, &\qquad \qquad 
y^3_{\zeta}  &=  U^{\rho}_{\sigma \zeta}  , &\qquad \qquad   
 y^3_{\ov{\zeta}}   &= U^{\rho}_{\sigma \ov{\zeta}}, \\ 
2 \, y^1_{\zeta} &= U^{\rho}_{\rho \zeta}, & \qquad \qquad
2 \, y^1_{\ov{\zeta}}   &=   U^{\rho}_{\rho \ov{\zeta}}, & \qquad \qquad
-y^4_{\tau} + y^8_{\sigma} & = U^{\zeta}_{\tau \sigma},\\
-y^2_{\tau} + y^8_{\rho} & =  U^{\zeta}_{\tau \rho},& \qquad \qquad 
- y^1_{\tau} + y^8_{\zeta} & = U^{\zeta}_{\tau \zeta} , & \qquad \qquad
y^8_{\ov{\zeta}} & =U^{\zeta}_{\tau \ov{\zeta}} , \\
-y^2_{\sigma} + y^4_{\rho}  & = U^{\zeta}_{\sigma \rho}, & \qquad \qquad 
-y^1_{\sigma} + y^4_{\zeta} & = U^{\zeta}_{\sigma \zeta}, & \qquad \qquad
y^4_{\ov{\zeta}} & = U^{\zeta}_{\sigma \ov{\zeta}}, \\
-y^1_{\rho} + y^2_{\zeta} & = U^{\zeta}_{\rho \zeta}, & \qquad \qquad 
y^2_{\ov{\zeta}}  & = U^{\zeta}_{\rho \overline{\zeta}}, & \qquad \qquad
 y^1_{\ov{\zeta}} & =  U^{\zeta}_{\zeta \overline{\zeta}}.
\end{alignat*} 
Eliminating $y^1_{\ov{zeta}}$ and $y^1_{\ov{\zeta}}$ among the previous equations leads to
the normalizations:
\begin{equation*}
\begin{aligned}
\bb & = \A \, {\bf B}_0, \\
\cc & = \A^2 \, {\bf C}_0, \\
\ff & = \A^3 \, {\bf F}_0,
\end{aligned}
\end{equation*}
where the functions ${\bf B}_0$, ${\bf C}_0$ and ${\bf F}_0$ are defined by:

\begin{equation*}
\begin{aligned}
{\bf B}_0 & := \frac{3i}{10} \frac{\Lb(B)}{B^{\frac{3}{2}}} - \frac{i}{5} \frac{A}{B^{\frac{1}{2}}}
- \frac{i}{10} \frac{K}{B^{\frac{1}{2}}}- \frac{i}{10} \frac{\LL(B)}{B^{\frac{1}{2}}}
,\\
{\bf C}_0 & := \frac{11}{20} \frac{\LL(B)}{B^{\frac{1}{2}}} + \frac{3}{20} \, B^{\frac{1}{2}}
G + \frac{1}{20} \, \frac{\Lb(B)}{B^{\frac{3}{2}}} - \frac{1}{5} \frac{A}{B^{\frac{1}{2}}} + \frac{3}{20} \frac{K}{B^{\frac{1}{2}}}
,\\
{\bf F}_0 & := 
\frac{1}{10} \, \frac{\LL(B)}{B} + \frac{3}{10}  \, B^{\frac{1}{2}} G + \frac{1}{10} \, \, \frac{\Lb(B)}{B^{\frac{3}{2}}}
- \frac{2}{5} \frac{A}{B^{\frac{1}{2}}} + \frac{3}{10} \, \frac{K}{B^{\frac{1}{2}}}
.\end{aligned}
\end{equation*}

The absorbed structure equations take the form:
\begin{equation*}
\begin{aligned}
d \tau 
&=
4 \, \widetilde{\beta}^1 \wedge \tau 
+
\frac{\mathfrak{I}_1}{\A} \left. \tau \wedge \zeta \right. 
-
\frac{\mathfrak{I}_1}{\A}\left. \tau \wedge \ov{\zeta} \right.
+
3  \, \frac{\mathfrak{I}_1}{\A} \left. \sigma \wedge \rho \right.
+
\left.\sigma \wedge \zeta \right.
+
\left. \sigma \wedge \ov{\zeta} \right.,
\\
d \sigma &= 3 \, \left. \widetilde{\beta}^1 \wedge \sigma \right. + \left. \widetilde{\beta}^6 \wedge \tau \right. 
- \frac{\mathfrak{I}_1}{2 \A} \left. \sigma \wedge \zeta \right.
+
\frac{\mathfrak{I}_1}{2 \A} \left. \sigma \wedge \ov{\zeta} \right.
+
\rho \wedge \zeta
+
\rho \wedge \ov{\zeta}
,\\
d \rho
&=
2 \widetilde{\beta}^{1} \wedge \rho + \widetilde{\beta}^3 \wedge \sigma + \widetilde{\beta}^7 \wedge \tau 
-
\frac{\mathfrak{I}_1}{2 \A}
\left. \rho \wedge \zeta \right.  
+
 \frac{\mathfrak{I}_1}{2 \A} \left. \rho \wedge \ov{\zeta} \right.
+
i \, \left. \zeta \wedge \ov{\zeta} \right.
,\\
d \zeta
& =
\widetilde{\beta^{1}} \wedge \zeta 
+ \widetilde{\beta^{2}} \wedge \rho + \widetilde{\beta^{4}} \wedge \sigma + \widetilde{\beta^8} \wedge \tau 
,
\end{aligned}
\end{equation*}

where the function
$\mathfrak{I}_1$
is a biholomorphic invariant of $M$ and is given by:
\begin{equation*}
\mathfrak{I}_1:=
\frac{1}{2} \, \frac{\LL(B)}{B} + \frac{3}{10}  \, B^{\frac{1}{2}} G - \frac{1}{10} \, \, \frac{\Lb(B)}{B^{\frac{3}{2}}}
+ \frac{2}{5} \frac{A}{B^{\frac{1}{2}}} - \frac{3}{10} \, \frac{K}{B^{\frac{1}{2}}}
.\end{equation*}

We introduce the coframe 
$\omega_2:=\left(\tau_2, \sigma_2, \rho_2, \zeta_2, \ov{\zeta}_2 \right)$ on $M$, defined by:
\begin{equation*}
\left \{
\begin{aligned}
\tau_2 & := \tau_1 \\
\sigma_2 &:= {\bf F}_0 \, \tau_1 + \sigma_1, \\
\rho_2 &:= \rho_1 + {\bf C}_0 \, \sigma_1, \\
\zeta_2 &:= \zeta_1 + {\bf B}_0 \, \rho_1, 
\end{aligned} \right.
\end{equation*}
and the subgroup $G_3 \subset G_2$: 
\begin{equation*}
G_3:= \left\{
\begin{pmatrix}
{\A^4}  & 0 & 0 & 0 & 0 \\
0  & \A^3 & 0 & 0 & 0 \\
\G & 0 & \A^2 & 0 & 0 \\
\hh & \dd & 0 & \A & 0 \\
\kk & \ee & 0 & 0 & \A
\end{pmatrix}
, \, 
\A \in \R \setminus \{0\}, \, \dd, \, \ee, \, \G, \, \hh, \, \kk,  \in \C \right \}.
\end{equation*}
We notice that $\sigma_2$ is a real one-form.
The normalizations:
\begin{equation*}
\bb:= \A \, {\bf B}_0, \qquad  \cc:= \A^2 \, {\bf C}_0, \qquad \ff:= \A^3 \, {\bf F}_0,
\end{equation*}
amount to
consider the subbundle $P^3 \subset P^2$ consisting of those coframes $\omega$ of the form
\begin{equation*}
\omega := g \cdot \omega_2, \qquad \text{where $g$ is a function} \,\,\,  
g: M \stackrel{g}{\longrightarrow}  G_3. 
\end{equation*}

A basis of the Maurer Cartan forms of $G_3$ is given by: 
\begin{equation*}
\begin{aligned}
\gamma^1  &: = {\frac {{ d\A}}{\A}}, \\
\gamma^2 & :=-{\frac {\dd { d\A}}{{\A}^{4}}} + {\frac {{
 d\dd}}{{\A}^{3}}}, \\
\gamma^3 & :=-\frac {  \ee d \A}{\A^4} +
\frac { d\ee}{\A^{3}}, \\
\gamma^4 & :=- 2 \, \frac {  \G d \A}{\A^5} +
\frac { d\G}{\A^{4}}, \\
\gamma^5 & :=- \frac {  \hh d \A}{\A^5} +
\frac { d\hh}{\A^{4}}, \\
\gamma^6 & :=- \frac {  \kk d \A}{\A^5} +
\frac { d\kk}{\A^{4}}.
\end{aligned}
\end{equation*}

We get the following absorbed structure equations for $P^3$:
\begin{multline*}
d \tau = 4 \, \widetilde{\gamma}^1 \wedge \tau 
+
\frac{\mathfrak{I}_1}{\A} \left. \tau \wedge \zeta \right. 
-
\frac{\mathfrak{I}_1}{\A}\left. \tau \wedge \ov{\zeta} \right.
+
3  \, \frac{\mathfrak{I}_1}{\A} \left. \sigma \wedge \rho \right.
+
\left.\sigma \wedge \zeta \right.
+
\left. \sigma \wedge \ov{\zeta} \right.,
\end{multline*}

\begin{multline*}
d \sigma = 3 \left. \widetilde{\gamma}^1 \wedge \sigma \right. \\
+
V^{\sigma}_{\tau \rho} \left. \tau \wedge \rho \right.
+
V^{\sigma}_{\tau \zeta} \left. \tau \wedge \zeta \right.
+
V^{\sigma}_{\tau \ov{\zeta}} \left. \tau \wedge \ov{\zeta} \right. 
+
V^{\sigma}_{\sigma \rho} \left. \sigma \wedge \rho \right.\\
- \frac{\mathfrak{I}_1}{2 \A} \left. \sigma \wedge \zeta \right.
+
\frac{\mathfrak{I}_1}{2 \A} \left. \sigma \wedge \ov{\zeta} \right.
+
\left. \rho \wedge \zeta \right.
+
\left. \rho \wedge \ov{\zeta} \right.
,\end{multline*}\begin{multline*}
d \rho
=
2 \left. \widetilde{\gamma}^{1} \wedge \rho \right. + \left. \widetilde{\gamma}^4 \wedge \tau \right. \\
+
V^{\rho}_{\sigma \rho} \, \sigma \wedge \rho
+
V^{\rho}_{\sigma \zeta} \, \sigma \wedge \zeta
+
V^{\rho}_{\sigma \ov{\zeta}} \, \sigma \wedge \ov{\zeta} \\
+
\frac{\mathfrak{I}_1}{2 \A}
+
\left. \rho \wedge \zeta \right.  
+
 \frac{\mathfrak{I}_1}{2 \A} \left. \rho \wedge \ov{\zeta} \right. 
+
i \, \left. \zeta \wedge \ov{\zeta} \right.
,\end{multline*}

\begin{multline*}
d \zeta
=
\left. \widetilde{\gamma}^{1} \wedge \zeta \right.+ \left. \widetilde{\gamma}^{2} \wedge \sigma \right. +
\left. \widetilde{\gamma}^5 \wedge \tau \right.
+
V^{\zeta}_{\rho \zeta} \left. \rho \wedge \zeta \right.
+
V^{\zeta}_{\rho \overline{\zeta}} \left. \rho \wedge \overline{\zeta} \right.
,\end{multline*}

From the essential torsion coefficients $V^{\sigma}_{\tau \zeta}$, $V^{\sigma}_{\tau \ov{\zeta}}$ and
$V^{\zeta}_{\rho \ov{\zeta}}$,
we obtain the normalizations:
\begin{equation*}
\dd :=  \A \, {\bf D_0}, \qquad   \qquad \G:= \A^2 \, {\bf G_0},
\end{equation*}
where
\begin{equation*}
{\bf D_0}:= i \, {\bf B}_0^2 - \frac{A {\bf B}_0}{B^{\frac{1}{2 }}} + \frac{\Lb\left( {\bf B}_0 \right)}{B^{\frac{1}{2}}}
+ \frac{1}{2} \, \frac{\Lb(B) {\bf B}_0}{B^{\frac{3}{2}}},
\end{equation*}
and
\begin{multline*}
{\bf G_0}:=
- \frac{1}{4} \, \frac{\LL(B)}{B^{\frac{1}{2}}} {\bf F}_0 - {\bf F}_0^2 + \frac{1}{2} \, B^{\frac{1}{2}} G {\bf F}_0  
- \frac{1}{2} \, B^{\frac{1}{2}} \LL \left( {\bf F}_0 \right)  + {\bf C}_0 {\bf F}_0 \\+ \frac{1}{2} \, F B 
+ \frac{1}{4} \, \frac{\Lb(B)}{B^{\frac{3}{2}}} {\bf F}_0
+
\frac{1}{2} \, \frac{K}{B^{\frac{1}{2}}} {\bf F}_0 - \frac{1}{2} \, \frac{\Lb\left( {\bf F}_0 \right)}{B^{\frac{1}{2}}} 
+ \frac{J}{2} - \frac{1}{2} \, \frac{A}{B^{\frac{1}{2}}} {\bf F}_0.
\end{multline*}

We introduce the coframe 
$\omega_3:=\left(\tau_3, \sigma_3, \rho_3, \zeta_3, \ov{\zeta}_3 \right)$ on $M$, defined by:
\begin{equation*}
\left \{
\begin{aligned}
\tau_3 & := \tau_2 \\
\sigma_3 &:= \sigma_2 \\
\rho_3 &:= \rho_2 + {\bf C}_0 \, \tau_2, \\
\zeta_3 &:= \zeta_2 + {\bf D}_0 \, \sigma_2, 
\end{aligned} \right.
\end{equation*}
and the subgroup $G_4 \subset G_3$: 
\begin{equation*}
G_4:= \left\{
\begin{pmatrix}
{\A^4}  & 0 & 0 & 0 & 0 \\
0  & \A^3 & 0 & 0 & 0 \\
0 & 0 & \A^2 & 0 & 0 \\
\hh & 0 & 0 & \A & 0 \\
\ov{\hh} & 0 & 0 & 0 & \A
\end{pmatrix}
,
 \, 
\A \in \R \setminus \{0\}, \, \hh \in \C 
\right \}
\end{equation*}
The normalizations:
\begin{equation*}
\dd :=  \A \, {\bf D_0}, \qquad   \qquad \G:= \A^2 \, {\bf G_0},
\end{equation*}
amount to
consider the subbundle $P^4 \subset P^3$ consisting of those coframes $\omega$ of the form
\begin{equation*}
\omega := g \cdot \omega_3, \qquad \text{where $g$ is a function} \,\,\,  
g: M \stackrel{g}{\longrightarrow}  G_4. 
\end{equation*}

A basis of the Maurer-Cartan forms
is given by:
\begin{equation*}
\begin{aligned}
\delta^1  &: = {\frac {{ d\A}}{\A}}, \\
\delta^2 &: =- \frac {  \hh d \A}{\A^5} +
\frac { d\hh}{\A^{4}}, \\
\end{aligned}
\end{equation*}
together with $\ov{\delta}^2$.

As for the previous step, we determine the structure equations of $P^4$ using formula (\ref{eq:lifted}).
We just write here the expression of $d \zeta$, as it provides a normalization of $\hh$:
\begin{equation*}
d \zeta = \widetilde{\delta}^1 \wedge \zeta + \widetilde{\delta}^2 \wedge \tau +
W^{\zeta}_{\sigma \rho} \left. \sigma \wedge \rho \right. +  W^{\zeta}_{\sigma \zeta} \left. \sigma \wedge \zeta \right. 
+ W^{\zeta}_{\sigma \zeta} \left. \sigma \wedge \ov{\zeta} \right.
,\end{equation*}
for some modified Maurer-Cartan forms $\widetilde{\delta}^1$, $\widetilde{\delta}^2$.

The essential torsion coefficient $W^{\zeta}_{\sigma \zeta}$ can be normalized to $0$,
which is equivalent to the normalization:
\begin{equation*}
\hh := \A \, {\bf H}_0,
\end{equation*}
where 
\begin{multline*}
{\bf H}_0 :=
- {\bf D}_0 \, {\bf F}_0 + {\bf C}_0 \, {\bf D}_0 - \frac{\LL(B)}{B^{\frac{1}{2}}} {\bf D}_0 - \frac{A}{B^{\frac{1}{2}}} 
{\bf D}_0 + {\Lb \left( {\bf D}_0 \right)}{B^{\frac{1}{2}}} +
i \, {\bf B}_0 \, {\bf D}_0 \\
- i \, {\bf B}^2_0 \, {\bf C}_0 +  \frac{A}{B^{\frac{1}{2}}} \, {\bf B}_0 \, {\bf C}_0
- \LL(A) \, {\bf B}_0 
- \frac{\Lb \left( {\bf B}_0 \right)}{B^{\frac{1}{2}}} \, {\bf C}_0 - \frac{1}{2} \,  \frac{\Lb(B)}{B^{\frac{3}{2}}} 
{\bf B}_0 \, {\bf C}_0.
\end{multline*}

Let $G_5$ be the $1$-dimensional Lie subgroup of $G_4$ whose elements $g$ are of the form:
\begin{equation*}
g := \begin{pmatrix}
\A^4 & 0 & 0 & 0 & 0 \\
0 & {\A}^3 & 0 & 0 & 0  \\
0 &0  & \A^2 & 0 & 0 \\
0 &0 & 0 & \A & 0 \\
0 &0 & 0 & 0 & \A
\end{pmatrix}, \qquad \A \in \R \setminus \{ 0 \},
\end{equation*}
and let $\omega_4:=\left(\tau_4, \sigma_4, \rho_4, \zeta_4, \ov{\zeta}_4\right)$ 
be the coframe defined on $M$ by:
\begin{equation*}
\sigma_4 := \sigma_3, \qquad \qquad
\rho_4 : = \rho_3,  \qquad   \qquad
\zeta_4 : = \zeta_3 + {\bf H}_0 \, \tau_3.
\end{equation*}
The normalization of $\hh$ is equivalent to the reduction
of $P^4$ to a subbundle $P^5$ consisting of those coframes $\omega$ on $M$ such that:
\begin{equation*}
\omega := g \cdot \omega_3, \qquad \text{where $g$ is a function} \,\,\,  
g: M \stackrel{g}{\longrightarrow}  G_4. 
\end{equation*}

The Maurer-Cartan forms of $G_5$ are spanned by: 
\begin{equation*}
\alpha:= \frac{d \A}{\A}.
\end{equation*}

Proceeding as in the previous steps, we determine the structure equations of $P^4$
which take the absorbed form:

\begin{dgroup} \label{eq:coframe}
\begin{equation*}
d \tau 
=
4 \, \Lambda \wedge \tau 
+
\frac{\mathfrak{I}_1}{\A} \left. \tau \wedge \zeta \right. 
-
\frac{\mathfrak{I}_1}{\A}\left. \tau \wedge \ov{\zeta} \right.
+
3  \, \frac{\mathfrak{I}_1}{\A} \left. \sigma \wedge \rho \right.
+
\left.\sigma \wedge \zeta \right.
+
\left. \sigma \wedge \ov{\zeta} \right.,
\end{equation*}
\begin{multline*}
d \sigma = 3 \, \Lambda \wedge \sigma  \\ 
+ \frac{\mathfrak{I}_2}{ \A ^3} \, \tau \wedge \rho 
+ \frac{\mathfrak{I}_3}{ \A ^2} \, \tau \wedge \zeta
+ \frac{\ov{\mathfrak{I}}_3}{ \A ^2} \, \tau \wedge \ov{\zeta}
+  
\frac{\mathfrak{I}_4}{ \A ^2} \, \sigma \wedge \rho \\
- \frac{\mathfrak{I}_1}{2 \A} \left. \sigma \wedge \zeta \right.
+
\frac{\mathfrak{I}_1}{2 \A} \left. \sigma \wedge \ov{\zeta} \right. 
 +
\rho \wedge \zeta
+
\rho \wedge \ov{\zeta}
,
\end{multline*}
\begin{multline*}
d \rho
=
2 \Lambda \wedge \rho \\
+
\frac{\mathfrak{I}_5}{ \A^5}
\left. \tau \wedge \sigma \right.  
+
\frac{\mathfrak{I}_6}{ \A^4}
\left. \tau \wedge \rho \right.
+
\frac{\mathfrak{I}_7}{ \A^3}
\left. \tau \wedge \zeta \right.
+
\frac{\ov{\mathfrak{I}}_7}{ \A^3}
\left. \tau \wedge \ov{\zeta} \right. 
+
\frac{\mathfrak{I}_8}{\A^3} \left. \sigma \wedge \rho \right.\\
+
\frac{\mathfrak{I}_9}{ \A^2}
\left. \sigma \wedge \zeta \right.
+
\frac{\ov{\mathfrak{I}}_9}{ \A^2}
\left. \sigma \wedge \ov{\zeta} \right.
-
\frac{\mathfrak{I}_1}{2 \A}
\left. \rho \wedge \zeta \right.  
+
 \frac{\mathfrak{I}_1}{2 \A} \left. \rho \wedge \ov{\zeta} \right.
+
i \, \left. \zeta \wedge \ov{\zeta} \right.
,
\end{multline*}
\begin{multline*}
d \zeta
=
\Lambda \wedge \zeta \\
+
\frac{\mathfrak{I}_{10}}{ \A^6}
\left. \tau \wedge \sigma \right.  
+
\frac{\mathfrak{I}_{11}}{ \A^5}
\left. \tau \wedge \rho \right.
+
\frac{\mathfrak{I}_{12}}{ \A^4}
\left. \tau \wedge \zeta \right.
+
\frac{\mathfrak{I}_{13}}{ \A^4}
\left. \tau \wedge \ov{\zeta} \right. 
\\
+
\frac{\mathfrak{I}_{14}}{ \A^4}
\left. \sigma \wedge \rho \right.
+
\frac{\mathfrak{I}_{15}}{ \A^3}
\left. \sigma \wedge \zeta \right.
,\end{multline*}
\end{dgroup}
where $\Lambda$ is a modified-Maurer Cartan form:
\begin{equation*}
\Lambda:= \frac{d \A}{\A} - X_{\tau}\, \tau - X_{\sigma} \, \sigma  -X_{\rho} \, \rho 
- X_{\zeta} \, \zeta \, -
X_{\overline{\zeta}} \, \overline{\zeta},
\end{equation*}
and where
\begin{equation*}
\mathfrak{I}_i, \quad i= 1 \dots 15,
\end{equation*}
are biholomorphic invariants of $M$.

The exterior derivative of $\Lambda$ can be determined by taking the exterior derivative of the four
previous equations which leads to the so-called Bianchi-Cartan's identities.
We obtain the fact that $d \Lambda$ does not contain any $2$-form involving the $1$-form $\Lambda$, namely:
\begin{equation}
\label{eq:final}
d \Lambda= \sum_{\nu \mu} 
X_{\nu \mu} \left. \nu \wedge \mu \right., \qquad \nu, \, \mu = \tau, \, \sigma, \, \rho,  \,  \zeta, \ov{\zeta}.
\end{equation}

\section{Cartan Connection} \label{connection}
We recall that the model for CR-manifolds belonging to general class ${\sf III_2}$
is the CR-manifold defined by the equations:
\begin{equation*}
{\sf N}: \qquad \qquad
\begin{aligned}
w_1 & = \ov{w_1} + 2 \, i \, z \ov{z}, \\
w_2 & = \ov{w_2} + 2 \, i \, z \ov{z} \left( z + \ov{z} \right), \\
w_3 & = \ov{w_3} + 2 i \, z \zb \left( z^2 + \frac{3}{2} \, z \zb + \zb^2 \right).
\end{aligned}
\end{equation*}
Its Lie algebra of infinitesimal CR-automorphisms is given by the following theorem:
\begin{theorem}{\cite{pocchiola2}.}
The model of the class ${\sf III_2}${\rm :}
\begin{equation*}
{\sf N}: \qquad \qquad
\begin{aligned}
w_1 & = \ov{w_1} + 2 \, i \, z \ov{z}, \\
w_2 & = \ov{w_2} + 2 \, i \, z \ov{z} \left( z + \ov{z} \right), \\
w_3 & = \ov{w_3} + 2 i \, z \zb \left( z^2 + \frac{3}{2} \, z \zb + \zb^2 \right),
\end{aligned}
\end{equation*}
has a ${\bf 6}$-dimensional Lie algebra of CR-automorphisms
${\sf aut_{CR}}({\sf N})  $. 
A basis for the Maurer-Cartan forms of ${\sf aut_{CR}}({\sf N})$ is provided
by the $6$ differential $1$-forms  $\tau$, $\sigma$, $\rho$, $\zeta$,  $\ov{\zeta}$, $\alpha$,
which satisfy the Maurer-Cartan equations:
\begin{equation*}
\begin{aligned}
d \tau &= 4 \left. \alpha \wedge \tau \right.
+ \left.\sigma \wedge \zeta \right.
+ \left. \sigma \wedge \ov{\zeta} \right., \\
d \sigma &= 3 \left. \alpha \wedge \sigma \right. + \left. \rho \wedge \zeta \right. 
+ \left.  \rho \wedge \ov{\zeta} \right., \\
d \rho &  =  2 \left. \alpha \wedge \rho \right. + i \, \left. \zeta \wedge \ov{\zeta} \right., \\
d \zeta &= \left. \alpha \wedge \zeta \right., \\
d \ov{\zeta} &= \left. \alpha \wedge \ov{\zeta} \right.,\\
d \alpha &= 0.
\end{aligned}
\end{equation*}

\end{theorem}

Let us write $\mathfrak{g}$ instead of ${\sf aut_{CR}}({\sf N})$ for the Lie algebra of inifinitesimal automorphisms of ${\sf N}$
and let 
$\left(e_{\alpha}, e_{\tau}, e_{\sigma}, e_{\rho}, e_{\zeta}, e_{\ov{\zeta}} \right)$ be the dual basis of the 
basis of Maurer-Cartan 1-forms: $\left(\alpha, \tau, \sigma, \rho, \zeta, \ov{\zeta} \right)$.
From the above structure equations, the Lie brackets structure of $\mathfrak{g}$ is given by:

\begin{alignat*}{3}
\big[ e_{\alpha} , e_{\tau} \big] & = -4 \, e_{\tau}, \qquad \qquad 
& \big[ e_{\sigma} , e_{\zeta} \big] & = - \, e_{\tau}, \qquad \qquad
& \big[ e_{\sigma} , e_{\ov{\zeta}} \big] & = - \, e_{\tau}, \qquad \qquad \\
\big[ e_{\alpha} , e_{\sigma} \big] & = -3 \, e_{\sigma}, \qquad \qquad 
& \big[ e_{\alpha} , e_{\rho} \big] & = -2 \, e_{\rho}, \qquad \qquad 
& \big[ e_{\alpha} , e_{\zeta} \big] & = - \, e_{\zeta}, \qquad \qquad \\
\big[ e_{\alpha} , e_{\ov{\zeta}} \big] & = - \, e_{\ov{\zeta}}, \qquad \qquad 
& \big[ e_{\rho} , e_{\zeta} \big] & = -e_{\sigma}, \qquad \qquad 
& \big[ e_{\rho} , e_{\ov{\zeta}} \big] & = - e_{\sigma},\qquad \qquad \\
 \big[ e_{\zeta} , e_{\ov{\zeta}} \big] & = - i \, e_{\rho}, 
\end{alignat*}
the remaining brackets being equal to zero.

We refer to \cite{Kobayashi}, p. 127-128, for the definition of a Cartan connection.
Let $\mathfrak{g}_0 \subset \mathfrak{g}$ be the subalgebra spanned by $e_{\alpha}$,
$\mathfrak{G}$ the connected, simply connected Lie group whose Lie algebra is $\mathfrak{g}$ and $\mathfrak{G}_0$ the closed
$1$-dimensional subgroup of $\mathfrak{G}$ generated by $\mathfrak{g}_0$. 
We notice that $\mathfrak{G}_0 \cong G_5$, so that $P^5$ is a principal bundle over $M$ with structure group
$\mathfrak{G}_0$, and that $\dim \mathfrak{G} / \mathfrak{G}_0 = \dim M  = 5$.

Let $\left( \Lambda, \tau, \sigma, \rho, \zeta, \ov{\zeta} \right)$ be the coframe of $1$-forms on $P^5$ 
whose structure equation are given by (\ref{eq:coframe}) -- (\ref{eq:final})
and $\omega$ the $1$-form on $P$ with values in $\mathfrak{g}$
defined by:
\begin{equation*}
\omega(X):= \Lambda (X) \, e_{\alpha} + \tau(X) \, e_{\tau} + \sigma(X) \, e_{\sigma} + \rho(X) \, e_{\rho} + \zeta(X) \, e_{\zeta},
+ \ov{\zeta} (X) \, e_{\ov{\zeta}}, 
\end{equation*} 
for $X \in T_pP^5$.
We have:

\begin{theorem}
$\omega$ is a Cartan connection on $P^5$.
\end{theorem} 

\begin{proof}
We shall check that the following three conditions hold:
\begin{enumerate}
\item{\label{it:vertical} $\omega(e_{\alpha}^*) = e_{\alpha}$, where $e_{\alpha}^*$ is the vertical 
vector field on $P^4$ generated by the action of $e_{\alpha}$, }
\item{\label{it:fibre} $R_{a}^* \, \omega = {\sf Ad}(a^{-1})\,  \omega$ for every $a \in \mathfrak{G}_0$,}  
\item{\label{it:isom} for each $p \in P^5$, $\omega_p$ is an isomorphism 
$T_p P^5 \stackrel{\omega_p}{\longrightarrow} \mathfrak{g}$}.
\end{enumerate}

Condition (\ref{it:isom}) is trivially satisfied as $\left( \Lambda, \tau, \sigma, \rho, \zeta, \ov{\zeta} \right)$
is a coframe on $P^5$ and thus defines a basis of $T_p^*P^5$ at each point $p$.

Condition (\ref{it:vertical}) follows simply from the fact that $\Lambda$ is a modified-Maurer Cartan form
on $P^5$: 
\begin{equation*}
\Lambda:= \frac{d \A}{\A} - X_{\tau}\, \tau - X_{\sigma} \, \sigma  -X_{\rho} \, \rho 
- X_{\zeta} \, \zeta \, -
X_{\overline{\zeta}} \, \overline{\zeta},
\end{equation*}
so that 
\begin{equation*}
\omega(e_{\alpha}^*) = \Lambda (e_{\alpha^*}) = e_{\alpha},
\end{equation*}
as 
\begin{equation*}
\tau(e_{\alpha}) = \sigma(e_{\alpha}^*)= \rho(e_{\alpha}^*)= \zeta(e_{\alpha}^*)= \ov{\zeta}(e_{\alpha}^*)= 0,
\qquad \frac{d \A}{\A} (e_{\alpha}^*) = 1,
\end{equation*}
since $e_{\alpha}^*$ is a vertical vector field on $P^5$.

Condition (\ref{it:fibre}) is equivalent to its infinitesimal counterpart:
\begin{equation*}
\LL_{e_{\alpha}^*} \, \omega = - {\sf ad}_{e_{\alpha}} \omega,
\end{equation*}
where 
$\LL_{e_{\alpha}^*} \, \omega$ is the Lie derivative of $\omega$ by the vector field
$e_{\alpha}^*$ and where ${\sf ad}_{e_{\alpha}}$ is the linear map 
$\mathfrak{g} \rightarrow \mathfrak{g}$ defined by:
${\sf ad}_{e_{\alpha}} (X) = \big[ e_{\alpha}, X \big].$
We determine $\LL_{e_{\alpha}^*} \, \omega$ with the help of Cartan's formula:
\begin{equation*}
\LL_{e_{\alpha}^*} \, \omega = e_{\alpha^*}  \, \lrcorner \,  d \omega + d \left( e_{\alpha}^* \, \lrcorner \, \omega \right)
,\end{equation*}
with
\begin{equation*}
d \left( e_{\alpha}^* \, \lrcorner \, \omega \right) = 0
\end{equation*}
from condition (\ref{it:vertical}).
The structure equations (\ref{eq:coframe})--(\ref{eq:final}) give:
\begin{equation*}
e_{\alpha^*}  \, \lrcorner \,  d \omega =  
\begin{pmatrix}
0 \\
4 \, \tau \\
3 \, \sigma \\
2 \, \rho \\
\zeta \\
\ov{\zeta}
\end{pmatrix}
,\end{equation*}
which is easily seen being equal to $ - {\sf ad}_{e_{\alpha}} \omega$ from the Lie bracket structure of
$\mathfrak{g}$.
\end{proof}

\end{document}